\documentclass[12pt,a4paper]{article}
  \def\SvZfontcode{8} 
  \def\SvZslantedGreekCapitals{1}
  \usepackage{amsmath,amsthm,amscd}

\usepackage[T1]{fontenc}

\def\SvZrequireslantedRedef{0}
\ifcase\SvZfontcode
\usepackage{bm}
\usepackage{amssymb}
\def\SvZrequireslantedRedef{1}
\or
\usepackage{fouriernc}
\usepackage{bm}
\def\SvZrequireslantedRedef{1}
\or
\usepackage{fourier}
\usepackage{bm}
\def\SvZrequireslantedRedef{1}
\or
\usepackage{amssymb}
\if\SvZslantedGreekCapitals 1
\usepackage[slantedGreek]{mathptmx}
\else
\usepackage{mathptmx}
\fi
\DeclareMathAlphabet{\bm}{OT1}{ptm}{b}{it} 
\or
\usepackage{kmath,kerkis}
\usepackage{bm}
\def\SvZrequireslantedRedef{1}
\or
\if\SvZslantedGreekCapitals 1
\usepackage[sc,slantedGreek]{mathpazo}
\else
\usepackage[sc]{mathpazo}
\fi
\usepackage{bm}
\or
\usepackage[adobe-utopia]{mathdesign}
\usepackage{bm}
\or
\usepackage[urw-garamond]{mathdesign}
\usepackage{bm}
\or
\usepackage[bitstream-charter]{mathdesign}
\usepackage{bm}
\or
\usepackage{lmodern}
\usepackage{bm}
\usepackage{amssymb}
\def\SvZrequireslantedRedef{1}
\or
\usepackage{arev}
\fi

\if\SvZrequireslantedRedef 1
\if\SvZslantedGreekCapitals 1
\renewcommand{\Gamma}{\varGamma}
\renewcommand{\Delta}{\varDelta}
\renewcommand{\Theta}{\varTheta}
\renewcommand{\Lambda}{\varLambda}
\renewcommand{\Xi}{\varXi}
\renewcommand{\Pi}{\varPi}
\renewcommand{\Sigma}{\varSigma}
\renewcommand{\Upsilon}{\varUpsilon}
\renewcommand{\Phi}{\varPhi}
\renewcommand{\Psi}{\varPsi}
\renewcommand{\Omega}{\varOmega}
\fi
\fi

\renewcommand{\phi}{\varphi}

\reversemarginpar

\usepackage{url}
\usepackage{graphicx}
\usepackage{color}

\ifx\SvZinPresentation\undefined
\usepackage{float}
\usepackage{longtable}
\usepackage[margin=10pt,labelfont=bf,textfont=it]{caption}

\usepackage[colorlinks,linkcolor=black,anchorcolor=black,citecolor=black,filecolor=black,menucolor=black,runcolor=black,urlcolor=black]{hyperref}

\fi

\newcommand{\mathds}{\mathbb}

\DeclareMathOperator{\rank}{rk}
\newcommand{\corank}{\rank^*}
\DeclareMathOperator{\si}{si}
\DeclareMathOperator{\co}{co}

\newcommand{\N}{\mathds{N}}

\newcommand{\ignore}[1]{}

\DeclareMathOperator{\closure}{cl}

\newcommand{\smin}{\ensuremath{-}}

\newcommand{\delete}{\ensuremath{\backslash}}
\newcommand{\contract}{\ensuremath{\!/}}

\DeclareMathOperator{\bw}{bw}

\ifx\SvZinPresentation\undefined
\newtheorem{theorem}{Theorem}[section]
\else
\newtheorem{theorem}{Theorem}
\fi

\newtheorem{lemma}[theorem]{Lemma}

\newtheorem{corollary}[theorem]{Corollary}
\newtheorem{claim}{Claim}[theorem] 

\theoremstyle{definition}
\newtheorem{definition}[theorem]{Definition}

\newenvironment{claimenv}{\list{}{\rightmargin0pt\leftmargin10pt\topsep0pt}\item[]}{\endlist}

\newenvironment{subproof}{\begin{claimenv}\begin{proof}}{\end{proof}\end{claimenv}}

\newcommand{\tangle}{\mathcal{T}}

\newcommand{\coclosure}{\closure^*}
\DeclareMathOperator{\fullclosure}{fcl}

\usepackage[numbers,longnamesfirst]{natbib}
\newcommand{\Dutchvon}[2]{#2}

\begin{document}
\title{Matroid 3-connectivity and branch width\footnote{The research for this paper was supported by the Natural Sciences and Engineering Research Council of Canada. The second author was also supported by the NWO (The Netherlands Organization for Scientific Research) free competition project ``Matroid Structure -- for Efficiency'' led by Bert Gerards, and by National Science Foundation grant 1161650.}}
\author{Jim Geelen\footnote{University of Waterloo, 200 University Avenue West, Waterloo, Ontario, Canada N2L 3G1.} \and Stefan H. M. van Zwam\footnote{Department of Mathematics, Louisiana State University, Baton Rouge, LA, United States. Email: \url{svanzwam@math.lsu.edu}}.}

\maketitle

\abstract{
We prove that, for each nonnegative integer $k$ and
each matroid $N$, if $M$ is a $3$-connected
matroid containing $N$ as a minor, and
the branch width of $M$ is sufficiently large,
then there is a $k$-element set $X\subseteq E(M)$ such that
one of $M\delete X$ and $M\contract X$ is 3-connected and
contains $N$ as a minor.
}

\section{Introduction} 
\label{sec:introduction}

We prove the following theorem.

\begin{theorem}\label{thm:mainresintro}
	Let $M$ be a matroid, let $N$ be a minor of $M$, and let $k$ be a nonnegative constant. If the branch width of $M$ is at least $20k + 2|E(N)|$, then there is a set $X\subseteq E(M)$ that has at least $k$ elements and is both independent and coindependent such that $M\delete X$ or $M\contract X$ is 3-connected with $N$ as a minor.
\end{theorem}

Our main result (Theorem \ref{thm:mainres})	is a strengthening of Theorem \ref{thm:mainresintro} that involves tangles. Theorem \ref{thm:mainresintro} can be seen as a generalization of the Splitter Theorem, proved by Seymour \cite{Sey80} and, independently, by Tan \cite{Tan81}. In particular, consider the following formulation:

\begin{theorem}[Splitter Theorem]\label{thm:splitter}
	Let $M$ be a 3-connected matroid, and $N$ a 3-connected proper minor of $M$. If $M$ is not a wheel or a whirl, then there is an $e \in E(M)$ such that one of $M\delete e$ and $M\contract e$ is 3-connected with a minor isomorphic to $N$.
\end{theorem}

When the minor $N$ is the empty matroid, this result is known as Tutte's Wheels and Whirls Theorem \cite{Tut66}. Several variants exist, such as \citep{OSW08b,OSW10}.

In the Splitter Theorem, the two obstructions to the existence of a removable element, the wheels and whirls, have branch width 2. The branch width of a matroid is minor-monotone, so an easy consequence of Theorem \ref{thm:splitter} is the following.

\begin{corollary}\label{cor:bwsplitter}
	Let $M$ be a 3-connected matroid with $\bw(M)\geq 3$, and $N$ a 3-connected proper minor of $M$. Then there is an $e \in E(M)$ such that one of $M\delete e$ and $M\contract e$ is 3-connected with a minor isomorphic to $N$.
\end{corollary}

Sometimes deleting one element is insufficient. For instance, in papers on stabilizers or excluded minors, the notion of a \emph{deletion pair} is central \cite{Whi96b,GGK,HMZ11,MWZ09}. In those papers, 3-connectivity cannot be guaranteed when two elements are removed, but the 2-separations that are introduced can be handled at the cost of a more complicated analysis. Our result generalizes Corollary \ref{cor:bwsplitter} by showing that, if the branch width is large enough, then we can either delete or contract any fixed number of elements and preserve both 3-connectivity and a specified minor.

Note that, rather than preserving a matroid isomorphic to the minor $N$, we preserve $N$ itself. Additionally, we impose fewer conditions on the connectivity of $N$. \citet[Theorem 11.1.2]{ox2} describes a version of the Splitter Theorem in which $N$ is not 3-connected, but the conclusion of that theorem is significantly weaker than in the 3-connected case. 

%
%

\paragraph{Notation.} Our notation and terminology follow \citet{ox2}. Additionally, if $\mathcal{X}$ is a collection of sets, $\cup\mathcal{X}$ denotes the union of all sets in $\mathcal{X}$, and $\cap\mathcal{X}$ the intersection.


\section{Connectivity, branch width, and tangles} 
\label{sec:definitions}

\subsection{Closure} 
\label{sub:closure}
We use the usual definitions of closure and coclosure from \citet{ox2}. In addition, we define the following. Let $M$ be a matroid and $X\subseteq E(M)$. We say that subset $X$ of the groundset of a matroid $M$ is \emph{fully closed} if $X$ is both closed and coclosed in $M$. The smallest fully closed set containing $X$ is denoted by $\fullclosure_M(X)$. Some more terminology: a \emph{line} is a closed set of rank two. A line is \emph{long} if it has at least three rank-one flats.

The following elementary lemma is \cite[Proposition 2.1.12]{ox2}.

\begin{lemma}\label{lem:closurecomplement}
	Let $M$ be a matroid, $e \in E(M)$, and $(A,B)$ a partition of $E(M)\smin e$. Then $e \in \closure_M(A)$ if and only if $e \not \in \coclosure_{M}(B)$.
\end{lemma}


\subsection{Connectivity and separations} 
\label{sub:connectivity}
\label{sub:separations}
  An unfortunate consequence of the graph-theoretic pedigree of matroid theory is that two definitions of the connectivity function coexist (differing from each other by an additive constant of 1). We will take the smaller of these definitions:
  \begin{definition}\label{def:conn}
  	Let $M$ be a matroid. The \emph{connectivity function} $\lambda_M: 2^{E(M)}\rightarrow \N$ is defined by
    \begin{align*}
    	\lambda_M(X) := \rank_M(X) + \rank_M(E(M)\smin X) - \rank(M).
    \end{align*}
  \end{definition}

We will use the following elementary properties of the connectivity function, which can be found in \cite[Section 8.2]{ox2}:  

\begin{lemma}\label{lem:connprop}
	Let $M$ be a matroid, and $X,Y \subseteq E(M)$. The connectivity function of $M$ has the following properties.
	\begin{enumerate}
		\item\label{it:con1} $\lambda_M(X) = \rank_M(X) + \rank_{M^*}(X) - |X|$;
		\item\label{it:con2} $\lambda_M(E(M)\smin X) = \lambda_M(X)$;
		\item\label{it:con3} $\lambda_{M^*}(X) = \lambda_M(X)$;
		\item\label{it:con4} If $e \in E(M)\smin X$, then $\lambda_{M\delete e}(X) \leq \lambda_M(X) \leq \lambda_{M\delete e}(X) + 1$;
		\item\label{it:con5} $\lambda_M(X) + \lambda_M(Y) \geq \lambda_M(X\cap Y) + \lambda_M(X\cup Y)$.
  \end{enumerate}
\end{lemma}

For ease of reference, we repeat the usual definitions of separations and connectivity before stating some less common results.

\begin{definition}\label{def:separating}
	Let $M$ be a matroid. A set $X\subseteq E(M)$ is \emph{$k$-separating} if $\lambda_M(X) < k$. It is \emph{exactly} $k$-separating if $\lambda_M(X) = k-1$. 
\end{definition}

\begin{definition}\label{def:separation}
	Let $M$ be a matroid, and let $(X,Y)$ be a partition of $E(M)$. If $|X|, |Y| \geq k$ and $\lambda_M(X) < k$, then $(X,Y)$ is a \emph{$k$-separation} of $M$. If $\lambda_M(X) = k-1$, then $(X,Y)$ is an \emph{exact} $k$-separation of $M$.
\end{definition}

\begin{definition}\label{def:connected}
	A matroid $M$ is \emph{$k$-connected} if $M$ has no $k'$-separations with $k' < k$.
\end{definition}

Two partitions $(X_1, Y_1)$ and $(X_2, Y_2)$ \emph{cross} if $X_1\cap X_2$, $X_1\cap Y_2$, $Y_1 \cap X_2$, $Y_1\cap Y_2$ are all nonempty. An application of the following lemma (from \cite{OSW04}) is called an \emph{uncrossing}. We omit the standard proof.

\begin{lemma}\label{lem:uncrossing}
	Let $M$ be a $k$-connected matroid, and let $X_1, X_2$ be $k$-separating sets.
	\begin{enumerate}
		\item \label{it:uncros1} If $|X_1\cap X_2| \geq k-1$, then $X_1\cup X_2$ is $k$-separating.
		\item \label{it:uncros2} If $|E(M)\smin (X_1\cup X_2)| \geq k-1$, then $X_1\cap X_2$ is $k$-separating.
	\end{enumerate}
\end{lemma}

Since we wish to preserve 3-connectivity, we have to know how separations change when taking minors.

\begin{lemma}\label{lem:3sepgutscoguts}
	Let $M$ be a $k$-connected matroid, $(X,Y)$ an exact $k$-separation of $M$, and $e\in X$, not a loop. The following are equivalent.
	\begin{enumerate}
		\item\label{it:guts1} $(X\smin e, Y)$ is a $(k-1)$-separation in $M\contract e$;
		\item\label{it:guts2} $e \in \closure_M(Y)\cap\closure_M(X\smin e)$;
		\item\label{it:guts3} $e\not\in\coclosure_M(Y) \cup \coclosure_M(X\smin e)$.
	\end{enumerate}
\end{lemma}

See \cite[Section 8.2]{ox2} for a proof.

\begin{lemma}\label{lem:3sepcontract}
	Let $M$ be a $k$-connected matroid, $(X,Y)$ a $k$-separation of $M$, and $e\in X$ such that $M\contract e$ is $k$-connected. Then $e \not \in \closure_M(Y)$.
\end{lemma}

\begin{proof}
	Suppose that, contrary to the claim, $e\in\closure_M(Y)$. If $e\in\closure_M(X\smin e)$, then, by Lemma \ref{lem:3sepgutscoguts}, $M\contract e$ is not $k$-connected, a contradiction. If $e\not\in\closure_M(X\smin e)$, then $\lambda_M(X\smin e) < \lambda_M(X)$. But $M$ is $k$-connected, a contradiction.
\end{proof}

In some of our proofs we will require that a minor $N$ of a matroid $M$ has no loops or coloops. The following easy lemma implies that this assumption is not overly restrictive:

\begin{lemma}\label{lem:connectedminor}
	Let $M$ be a connected matroid and let $N$ be a minor of $M$. If $N$ has $l$ elements each of which is a loop or coloop, then $M$ has a minor $N'$ such that $N$ is a minor of $N'$, such that $N'$ has no loops and coloops, and such that $|E(N')| \leq |E(N)| + l$.
\end{lemma}

\begin{proof}
	Let $M$ be a connected matroid, let $N$ be a minor of $M$, and let $C, D \subseteq E(M)$ be such that $N = M \contract C \delete D$ with $C$ independent and $D$ coindependent. Let $e$ be a loop of $N$. Since $M$ is connected, $e$ is not a loop of $M$. Hence there is a circuit $X \subseteq C\cup e$ using $e$ with $|X| \geq 2$. Let $f\in X\smin e$, and consider $N'' := M\contract (C\smin f) \delete D$. Since $\{e,f\}$ is a parallel pair in $N''$, the matroid $N''$ has strictly fewer loops than $N$. Moreover, $|E(N'')| = |E(N)| + 1$. The result now follows by duality and induction.
\end{proof}

We note that Lemos and Oxley \cite{LO98} proved that, if $N$ has $k$ components, then $M$ has a \emph{connected} minor $N'$ on at most $|E(N)| + 2k - 2$ elements.

\subsection{2-separations}
In this subsection we consider preserving a minor in the presence of a 2-separation. The following lemma is a special case of \cite[Corollary 8.2.2]{ox2}.

\begin{lemma}\label{lem:simplecosimple2sep}
	Let $(A,B)$ be a 2-separation of a connected matroid $M$. If $|B| = 2$, then $B$ is a parallel or series pair.
\end{lemma}

\begin{lemma}\label{lem:2sepminorcondel}
	Let $M$ be a matroid, $N$ a minor of $M$, and $(A,B)$ a 2-separation of $M$ with $B\cap E(N) = \emptyset$. Then one of $M\delete B$ and $M\contract B$ has $N$ as a minor.
\end{lemma}

\begin{proof}
	Since no element of $B$ is in $N$, there are disjoint sets $C,D \subseteq B$ such that $B = C\cup D$ and $M\contract C \delete D$ has $N$ as a minor. If $\lambda_{M\delete D}(C) = 0$, then $M\delete D \contract C = M\delete D \delete C$ and the result follows. Therefore $\lambda_{M\delete D}(C) = \rank_M(A) + \rank_M(C) - \rank_M(A\cup C) = 1$. But
	\begin{align*}
		\lambda_{M\contract C}(A) & = \rank_M(A\cup C) - \rank_M(C) + \rank_M(B) - \rank_M(C) - (\rank(M) - \rank_M(C))\\
		                          & = \rank_M(A) + \rank_M(C) - 1 - \rank_M(C) + \rank_M(B)  - \rank(M) \\
		  & = \lambda_M(A) - 1 = 0,
	\end{align*}
	so $D$ is a separator of $M\contract C$. Hence $M\contract C \delete D = M\contract C \contract D$, and the result follows.
\end{proof}

An easy consequence is this:

\begin{corollary}\label{lem:2sepminorconoutclosure}
	Let $M$ be a matroid, $N$ a minor of $M$, and $(A,B)$ a 2-separation of $M$ with $B\cap E(N) = \emptyset$. If $e\in B\smin\closure_M(A)$, then $M\contract e$ has $N$ as a minor.
\end{corollary}

\begin{proof}
	If $M\contract B$ has $N$ as a minor, then we are done, so we may assume $M\delete B$ has $N$ as a minor. Consider $M' := M\delete (B\smin e)$. Since $e\not\in\closure_{M'}(A)$, it is a coloop of $M'$, and therefore $M'\delete e = M'\contract e$.
\end{proof}

We immediately find the following:

\begin{corollary}\label{cor:2sepminordelcon}
	Let $M$ be a matroid, $N$ a minor of $M$, and $(A,B)$ a 2-separation of $M$ with $B\cap E(N) = \emptyset$. If $e\in B\smin (\closure_M(A)\cup\coclosure_M(A))$, then both $M\delete e$ and $M\contract e$ have $N$ as a minor.
\end{corollary}

The following is \cite[Lemma 8.3.3]{ox2}.

\begin{lemma}\label{lem:circsacross2sep}
  Let $M$ be a matroid, and let $(A,B)$ be a 2-separation of $M$.	If $C_1, C_2$ are circuits of $M$, both of which meet both $A$ and $B$, then $(C_1\cap A)\cup (C_2\cap B)$ is a circuit of $M$.
\end{lemma}

To use Corollary \ref{cor:2sepminordelcon} effectively, we need a little more information about $\closure_M(A)\cup\coclosure_M(A)$. We omit the proof of the following lemma, which is straightforward with Lemma \ref{lem:circsacross2sep}.

\begin{lemma}\label{lem:2sepclosurecoclosuresmall}
	Let $M$ be a connected matroid, and $(A,B)$ a 2-separation of $M$. Then at least one of $\closure_M(A)\cap B$ and $\coclosure_M(A)\cap B$ is empty.
\end{lemma}

Next we consider the case in which $E(N)$ intersects $B$ in exactly one element. 

\begin{lemma}\label{lem:2sepminorintersect}
	Let $M$ be a connected matroid, $N$ a minor of $M$ with no loops and coloops, and $(A,B)$ a 2-separation of $M$ with $B\cap E(N) = \{f\}$. If $f$ is not in series or in parallel with any other element in $M$, then there exists an element $e\in B\smin f$ such that $M\delete e$ and $M\contract e$ both contain $N$ as a minor.	
\end{lemma}

\begin{proof}
	Suppose the lemma is false, and consider a counterexample with $|B|$ minimal. Let $C$ and $D$ be disjoint subsets of $B\smin f$ such that $M\contract C \delete D$ has $N$ as a minor, and pick $e \in D$. If there is a circuit $X$ of $M\contract e$ using $f$ and at least one element of $A$, then Lemma \ref{lem:circsacross2sep} implies that $M\contract e$ has $N$ as a minor. 
	
	Hence there is a separation $(A', B')$ of $M\contract e$ such that $E(N)\smin f\subseteq A'$ and $f \in B'$. Thus $(A'\cup e, B')$ is a 2-separation of $M$. By uncrossing with $(A,B)$, it follows that $(A'\cup A \cup e, B'\cap B)$ is a 2-separation for $M$. But this contradicts the minimality of $B$.
	
	The only remaining possibility is that $D$ is empty. But then, by duality, also $C = \emptyset$, a contradiction.
\end{proof}


\subsection{3-connectivity and fans} 
\label{ssec:3-connectivity}
Recall the following lemma by Bixby:

\begin{lemma}[{\citet{Bix82}; see also \citet[{Proposition 8.7.3}]{ox2}}]\label{lem:bixby}
	Let $M$ be a 3-connected matroid, and $e \in E(M)$. Then at least one of $M\contract e$ and $M\delete e$ has no non-minimal 2-separations.
\end{lemma}

Recall that a set $T\subseteq E(M)$ is a \emph{triangle} if $M|T \cong U_{2,3}$, and a \emph{triad} if it is a triangle of $M^*$. 

\begin{lemma}[{Tutte's Triangle Lemma; see \citet[{Lemma 8.7.7}]{ox2}}]\label{lem:Tuttriang}
	Let $M$ be a 3-connected matroid with $|E(M)| \geq 4$, and let $T = \{e,f,g\}$ be a triangle such that neither $M\delete e$ nor $M\delete f$ is 3-connected. Then $M$ has a triad containing $e$ and exactly one of $f$ and $g$.
\end{lemma}

Tutte's Triangle Lemma naturally leads to the notion of a \emph{fan}:

\begin{definition}\label{def:fan}
	Let $M$ be a matroid, and $F = (x_1, x_2, \ldots, x_k)$, $k\geq 3$, an ordered set of distinct elements of $E(M)$. We say that $F$ is a \emph{fan} of $M$ if $\{x_1,x_2,x_3\}$ is either a triangle or a triad, and for each $i \in \{1, \ldots, k-3\}$, if $\{x_i, x_{i+1},x_{i+2}\}$ is a triangle, then $\{x_{i+1}, x_{i+2},x_{i+3}\}$ is a triad, and if $\{x_i, x_{i+1},x_{i+2}\}$ is a triad, then $\{x_{i+1},x_{i+2},x_{i+3}\}$ is a triangle.
\end{definition}

A few trivial observations:

\begin{lemma}\label{lem:fanprops}
	Let $F = (x_1, x_2,\ldots,x_k)$ be a fan of a matroid $M$.
	\begin{enumerate}
		\item $F$ is a fan of $M^*$, with triangles and triads exchanged;
		\item $(x_k, x_{k-1},\ldots, x_1)$ is a fan of $M$;
		\item If $X \subseteq E(M)$ is fully closed, and $F$ is a maximal fan contained in $X$, then $F$ is a maximal fan in $E(M)$;
		\item If $k \geq 4$, and $1< l < k$, then neither $M\delete x_l$ nor $M\contract x_l$ is 3-connected.
	\end{enumerate}
\end{lemma}

The following lemma is due to \citet{OW00}.

\begin{lemma}\label{lem:endoffanOW}
	Let $M$ be a 3-connected matroid that is not a wheel or a whirl, and let $F$ be a maximal fan of $M$ with $k \geq 3$ elements. Then the elements of $F$ can be ordered $(x_1, x_2, \ldots, x_k)$ such that $(x_1,x_2,\ldots,x_k)$ is a fan, one of $M\delete x_1,M\contract x_1$ is 3-connected, and one of $M\delete x_k, M\contract x_k$ is 3-connected.
\end{lemma}

Note that, in a fan of length at least 4, the ends of the fan, $x_1$ and $x_k$, are the same for any ordering, and, in a fan of length at least 5, the order is completely fixed. We will upgrade Oxley and Wu's result so that we can preserve a minor, at the cost of a slightly worse bound on the size:

\begin{lemma}\label{lem:endoffan}
	Let $M$ be a 3-connected matroid that is not a wheel or a whirl, let $N$ be a minor of $M$ without loops or coloops, and let $F = (x_1,x_2,\ldots,x_k)$ be a maximal fan of $M$ with $k\geq 4$ elements. If $|E(N)\cap F| \leq 1$, then one of $M\delete x_1, M\contract x_1, M\delete x_k, M\contract x_k$ is 3-connected with $N$ as a minor.
\end{lemma}

\begin{proof}
	Suppose the theorem fails. To simplify notation we will assume $k$ to be even, leaving the analogous case for odd $k$ to the reader. By reversing the fan if necessary, we may assume $x_1 \not\in E(N)$. By dualizing $M$ and $N$ if necessary, we may assume that $\{x_1,x_2,x_3\}$ is a triangle (and therefore that $\{x_{k-2},x_{k-1},x_k\}$ is a triad). Hence $M\delete x_1$ is 3-connected. Suppose $M\delete x_1$ does not have $N$ as a minor. Then $M\contract x_1$ has $N$ as a minor. Let $M' := M\contract x_1$. The set $F\smin x_1$ is 2-separating in $M'$. 
	
	\begin{claim}
		$x_k \in E(N)$.
	\end{claim}
	
	\begin{subproof}
		Suppose this is not the case. Note that $F\smin \{x_1,x_k\}$ is a separator of $M'\delete x_k$. First, if $E(N)\cap (F\smin\{x_1,x_k\}) = \{f\}$, then $M'\delete x_k$ cannot have $N$ as a minor, since in such a minor $f$ would be either a loop or a coloop. Hence $M'\contract x_k$ has $N$ as a minor.	Next, if $E(N)\cap (F\smin x_1) = \emptyset$, then Corollary \ref{lem:2sepminorconoutclosure} implies that $M'\contract x_k$ has $N$ as a minor. In both cases it follows that $M\contract x_k$ has $N$ as a minor. But that matroid is 3-connected, and the result holds.
	\end{subproof}
	
	\begin{claim}
		$M' \contract x_2$ has $N$ as a minor.
	\end{claim}
	
	\begin{subproof}
		Note that $\{x_2,x_3\}$ form a parallel pair in $M'$. If $k = 4$, then $x_k$ is a coloop in $M'\delete \{x_2,x_3\}$, so $M'\contract x_2$ has $N$ as a minor.  If $k > 4$, then $\{x_2,x_3,x_4\}$ is a 2-separating set in $M'$ disjoint from $E(N)$. Moreover, we have $x_2 \in \coclosure_{M'}(\{x_3,x_4\})$, so by Lemma \ref{lem:closurecomplement}, $x_2 \not \in \closure_{M'}(E(M')\smin \{x_2,x_3,x_4\})$. From Corollary \ref{lem:2sepminorconoutclosure} it then follows that $M'\contract x_2$ has $N$ as a minor. 		
	\end{subproof}

	Therefore $M\contract x_2$ has $N$ as a minor. In that matroid $x_1$ and $x_3$ are in parallel, from which it follows that $M\delete x_1$ has $N$ as a minor, a contradiction.
\end{proof}


\subsection{Tangles and their matroids} 
\label{sub:tangles}

Instead of using branch width directly, we will use the notion of a \emph{tangle}, first defined by \citet{RSX} for hypergraphs, and extended to matroids by \citet{Dha96}. 
Our definitions follow \citet{GGRW06}.

\begin{definition}\label{def:tangle}
	Let $M$ be a matroid, and $\tangle$ a collection of subsets of $E(M)$. Then $\tangle$ is a \emph{tangle of order $\theta$} of $M$ if
	\begin{enumerate}
		\item\label{it:tan1} For all $X \in \tangle$, $\lambda_M(X) < \theta$;
		\item\label{it:tan2} For all $X\subseteq E(M)$ with $\lambda_M(X) < \theta$, either $X \in \tangle$ or $E(M)\smin X \in \tangle$;
		\item\label{it:tan3} If $X,Y,Z \in \tangle$, then $X\cup Y \cup Z \neq E(M)$;
		\item\label{it:tan4} For each $e \in E(M)$, $E(M)\smin e \not\in\tangle$.
	\end{enumerate}
\end{definition}
For instance, the empty set is a tangle of order $0$ of any nonempty matroid. The collection of all subsets of rank at most 2 is a tangle of order 3 of $\mathrm{PG}(2,q)$ for $q > 2$. For $q = 2$, condition $\eqref{it:tan3}$ is not satisfied. One can check that the maximum order of a tangle of a wheel or whirl is 2.

The following theorem, which was implicit in \citet{RSX}, shows that tangles and branch width are closely related. A proof using the definition of tangle given above can be found in \citet{GGRW06}. Note that they stated and proved the result for arbitrary connectivity functions.

\begin{theorem}\label{thm:bwtangle}
	Let $M$ be a matroid. The branch width of $M$ is one more than the maximum order $\theta$ of a tangle of $M$.
\end{theorem}

Because of this result, there is no need to define branch width here. We continue with some basic tangle facts, which can easily be deduced from the definition:

\begin{lemma}\label{lem:tangleprops}
	Let $M$ be a matroid, and $\tangle$ a tangle of $M$ of order $\theta$. 
	\begin{enumerate}
		\item\label{it:tanprop1} If $X \in \tangle$ and $X'\subseteq X$ is such that $\lambda_M(X') < \theta$, then $X' \in \tangle$;
		\item\label{it:tanprop2} If $\theta' < \theta$, and $\tangle' = \{X \in \tangle : \lambda_M(X) < \theta'\}$, then $\tangle'$ is a tangle of $M$ of order $\theta'$;
  	\item\label{it:tanprop3} $\tangle$ is a tangle of order $\theta$ of $M^*$.
	\end{enumerate}
\end{lemma}

Tangles can be helpful in dealing with crossing separations.

\begin{lemma}\label{lem:crosstangle}
	Let $M$ be a matroid, $\tangle$ a tangle of order $\theta$, and $X,Y\in \tangle$. If $\lambda_M(X\cup Y) < \theta$, then $X\cup Y \in \tangle$.
\end{lemma}

\begin{proof}
	Let $Z := E(M)\smin (X\cup Y)$. Either $X\cup Y \in \tangle$ or $Z \in \tangle$, by \ref{def:tangle}\eqref{it:tan2}. But if $Z \in \tangle$, then $X\cup Y \cup Z = E(M)$, contradicting \ref{def:tangle}\eqref{it:tan3}.
\end{proof}

We will apply this lemma regularly. In the case $Y = \{e\}$ we may do so without referring to it.

A useful means for studying tangles is the \emph{tangle matroid}. The following result is from \citet{GGRW06}:
\begin{theorem}\label{thm:tanglematroid}
	Let $M$ be a matroid, and $\tangle$ a tangle of $M$ of order $\theta$. Let $\rho:2^{E(M)}\rightarrow \N$ be defined by
	\begin{align*}
		\rho(X) := \left\{\begin{array}{ll}
			\min \{ \lambda_M(Y) : X \subseteq Y \in \tangle \} & \textrm{ if there is a } Y \text{ with } X \subseteq Y \in \tangle\\   
			\theta & \textrm{ otherwise}.
		\end{array}\right.
	\end{align*}
	Then $\rho$ is the rank function of a matroid.
\end{theorem}

We will denote this matroid by $M(\tangle)$, and write $\rank_\tangle, \closure_\tangle, \dots$ as shorthand for $\rank_{M(\tangle)}, \closure_{M(\tangle)}, \dots$. We will often work with independent sets in the tangle matroid, and we refer to them as $\tangle$-independent for short.

\begin{lemma}\label{lem:tangleindep}
	Let $M$ be a matroid, $\tangle$ a tangle of $M$ of order $\theta$, and $X$ a set that is independent in $M(\tangle)$. Then $X$ is both independent and coindependent in $M$.
\end{lemma}

\begin{proof}
	Suppose $X$ is not independent in $M$. Then $\lambda_M(X) \leq \rank_M(X) < |X|$. Since $|X| \leq \theta$, Definition \ref{def:tangle}\eqref{it:tan2} implies that either $X$ or its complement is in $\tangle$. From repeated application of Lemma \ref{lem:crosstangle}, starting from the singleton subsets of $X$, we conclude that $X \in \tangle$, and therefore $\rank_\tangle(X) \leq \lambda_M(X) < |X|$, a contradiction to the fact that $X$ is $\tangle$-independent. The result now follows by duality.
\end{proof}

%

If $N$ is a minor of $M$, then we can derive a tangle of $N$ from a tangle of $M$, as follows.

\begin{lemma}\label{lem:tangleminor}
	Let $M$ be a matroid, and $N$ a minor of $M$ such that $E(M) \smin E(N) = S$. Let $\tangle$ be a tangle of $M$ of order $\theta$. Define
	\begin{align*}
		\tangle' := \{X\smin S : X \in \tangle, \lambda_{N}(X \smin S) < \theta-|S|\}.
	\end{align*}
	 Then $\tangle'$ is a tangle of $N$ of order $\theta - |S|$.
\end{lemma}

\begin{proof}
  We give the proof if $S = \{e\}$. The result then follows by induction. The result is trivial if $\theta \leq 1$, since that implies $\tangle' = \emptyset$. Hence we may assume $\theta \geq 2$. 
	
	 Note that \ref{def:tangle}\eqref{it:tan1} follows immediately from our definition. For \ref{def:tangle}\eqref{it:tan2}, if $(X,Y)$ is $k$-separating in $N$ with $k \leq \theta - 1$, then Lemma \ref{lem:connprop}\eqref{it:con4} implies that $(X\cup e, Y)$ is $(k+1)$-separating in $M$, and hence either $X\cup e \in \tangle$ or $Y\in\tangle$. Then it follows immediately that $X \in \tangle'$ or $Y \in \tangle'$ respectively. For \ref{def:tangle}\eqref{it:tan3}, note that $\lambda_M(X\cup e) \leq \lambda_M(X)+1 < \theta$, so $(X\cup e) \cup Y \cup Z$ does not cover $E(M)$. Hence $X\cup Y \cup Z$ cannot cover $E(N)$. Finally, suppose $E(N)\smin f \in \tangle'$ for some $f \in E(N)$. Then we must have $E(M)\smin\{e,f\}\in\tangle$. But we also have $\{e\},\{f\} \in \tangle$, contradicting \ref{def:tangle}\eqref{it:tan3}.
\end{proof}

We say $\tangle'$ is the tangle \emph{inherited from} $\tangle$. We note some elementary properties of the corresponding tangle matroid:

\begin{lemma}\label{lem:tanglematroidminor}
	Let $M$ be a matroid, $\tangle$ a tangle of $M$ of order $\theta$, and $N$ a minor of $M$ with $E(M)\smin E(N) = \{e\}$. Let $\tangle'$ be the tangle of $N$ inherited from $\tangle$, and let $Z\subseteq E(N)$.
	\begin{enumerate}
		\item\label{it:tanmat1} $\rank_{\tangle}(Z) - 1 \leq \rank_{\tangle'}(Z) \leq \rank_\tangle(Z)$;
		\item\label{it:tanmat2} If $e\not\in \closure_\tangle(Z)$ and $\rank_\tangle(Z) < \theta$, then $\rank_{\tangle'}(Z) = \rank_{\tangle}(Z)$.
	\end{enumerate}
\end{lemma}

\begin{proof}
	Part \eqref{it:tanmat1} is a straightforward consequence of  \ref{lem:connprop}\eqref{it:con4}. Suppose Part \eqref{it:tanmat2} is false. Let $Z' \supseteq Z$ be such that $Z' \in\tangle'$ and $k = \lambda_N(Z') < \rank_\tangle(Z)$. By dualizing $M$ and $N$ if necessary we may assume $N = M\contract e$. Since $\lambda_M(Z') > k$, we must have $e\in\closure_M(Z')\cap\closure_M(E(M)\smin (Z'\cup e))$. But then $\lambda_M(Z'\cup e) = k+1$, and therefore $\rank_\tangle(Z\cup e) \leq k+1 \leq \rank_\tangle(Z)$. But this implies $e\in \closure_{\tangle}(Z)$, a contradiction.
\end{proof}

An easy corollary is the following.

\begin{lemma}\label{lem:indeptangleminor}
	Let $M$ be a matroid, let $\tangle$ be a tangle of $M$, let $X$ be a $\tangle$-independent subset of $E(M)$, and let $e\in X$. Then $X\smin e$ is $\tangle'$-independent in $M\delete e$, where $\tangle'$ is the tangle of $M\delete e$ inherited from  $\tangle$.
\end{lemma}

\begin{proof}
	Assume the result is false. Then there is a set $Z \supseteq X\smin e$ with $Z\in\tangle'$ and $\lambda_{M\delete e}(Z) < |X\smin e|$. By definition of $\tangle'$, either $Z\in\tangle$ or $Z\cup e \in \tangle$. By Lemma \ref{lem:connprop}\eqref{it:con4} we have that $\lambda_M(Z\cup e) \leq \lambda_{M\delete e}(Z) + 1 < |X|$. It follows that $Z\cup e \in \tangle$, because otherwise its complement together with $Z$ and $\{e\}$ would cover $E(M)$. But $X \subseteq Z\cup e$, a contradiction to $X$ being $\tangle$-independent.
\end{proof}

\begin{lemma}\label{lem:indepclosuretan}
	Let $M$ be a matroid, let $\tangle$ be a tangle of $M$ of order $\theta$, and let $X\subseteq E(M)$ be $\tangle$-independent. Let $Y := \closure_\tangle(X)$. If $e\in Y\smin X$, then $e \in \closure_M(Y\smin e)$ or $e\in \coclosure_M(Y\smin e)$.
\end{lemma}

\begin{proof}
	Assume $\rank_\tangle(X) = t$. If $t = \theta$ then $Y = E(M)$, and the result follows. Otherwise, we have $Y \in \tangle$, and $\lambda_M(Y) = t$, and $\lambda_M(Y\smin e) \geq t$, by the definition of $\rho$ in Theorem \ref{thm:tanglematroid}. Suppose $e$ is in neither the closure nor the coclosure of $Y\smin e$. Then
	\begin{align*}
		\lambda_M(Y) & = \rank_M(Y) + \corank_M(Y) - |Y| \\
		& = \rank_M(Y\smin e) + 1 + \corank_M(Y\smin e) + 1 - (|Y\smin e| + 1)\\
		& = \lambda_M(Y\smin e) + 1 > t,
	\end{align*}
	a contradiction.
\end{proof}

\begin{lemma}\label{lem:skew3seps}
  Let $M$ be a 3-connected matroid, $\tangle$ a tangle of $M$, and $X, X'$ long lines of $M(\tangle)$ such that $\rank_\tangle(X\cup X') = 4$. Let $e\in X$ and $M'\in\{M\delete e, M\contract e\}$ be such that $M'$ is 3-connected. Let $\tangle'$ be the tangle of $M'$ inherited from $\tangle$. Then $X'$ is closed in $M(\tangle')$.
\end{lemma}

\begin{proof}
	Since $\tangle$ is a tangle of $M^*$ we may dualize as necessary and assume $M' = M\contract e$. Suppose there is a $Z \in \tangle'$ with $X' \subsetneq Z$, and $\rank_{\tangle'}(Z) = 2$. Define $Y := E(M)\smin Z$. Then $(Z,Y\cup e)$ and $(Z\cup e, Y)$ are 4-separations of $M$, with $e\in\closure_M(Z)\cap\closure_M(Y)$, by Lemma \ref{lem:3sepgutscoguts}. Since $e\not\in\closure_M(E(M)\smin X)$ by Lemma \ref{lem:3sepcontract}, $Z\cap X$ and $Y\cap X$ are both nonempty. Let $e'\in Z\cap X$. Then $\rank_\tangle(X'\cup\{e,e'\}) \leq 3$, since $Z\cup e$ is 4-separating. But $\closure_{\tangle}(\{e,e'\}) = X$, so $3 \geq \rank_\tangle(X\cup \{e,e'\}) = \rank_\tangle(X\cup X') = 4$, a contradiction.
\end{proof}



%
%

\section{Finding elements to remove} 
\label{sec:finding_an_element_to_remove}
As a first step towards our result we show that, if the branch width is high enough, we can remove a single element and preserve 3-connectivity and $N$ as a minor.

\begin{theorem}\label{thm:contractunprotected}
	Let $M$ be a 3-connected matroid, let $N$ be a minor of $M$ without loops or coloops, let $\tangle$ be a tangle of $M$ of order at least 3,
	let $X$ be a long line in $M(\tangle)$, and let $f\in X$.
	If $(X\smin f)\cap E(N) = \emptyset$, then there
	exists an $e\in X\smin f$ such that either
	$M\delete e$ or $M\contract e$ is $3$-connected with $N$ as a minor.	
\end{theorem}

\begin{proof}
	First observe that, if $F \subseteq X$ is a fan, and $F'$ is a fan properly containing $F$, then $F' \subseteq X$, by Lemma \ref{lem:crosstangle}. If $X$ contains a fan of length 4 or more, then the result follows from Lemma \ref{lem:endoffan}. Therefore we can assume that $X$ contains no fans of length at least 4.
	
	Next, assume that there is an element $e \in X$ such that both $M\delete e$ and $M\contract e$ have $N$ as a minor. By Lemma \ref{lem:bixby}, either $\si(M\contract e)$ or $\co(M\delete e)$ is 3-connected. By duality we may assume the former. If $M\contract e$ is simple, then the result follows, so $e$ is on a triangle $T$. If possible, choose $T$ so that $|T\cap X| \geq 2$. Say $T = \{e,g,h\}$. First, assume $g \in X$. Then $h \in \closure_M(\{e,g\}) \subseteq \closure_M(X)$, so $\{e,g,h\}\subseteq X$. Assume $g \neq f$. Lemma \ref{lem:Tuttriang} implies that at least two of $\{e,g,h\}$ can be deleted keeping 3-connectivity. If $e$ is one of them the result follows. Otherwise, since $\{g,h\}$ is a parallel pair in $M\contract e$, it follows that $M\contract e \delete g$ has $N$ as a minor. But then $M\delete g$ has $N$ as a minor and is 3-connected. Hence we may assume that $g, h \in E(M)\smin X$. But then $e \in \closure_M(E(M)\smin X)$, and $(X - e, E(M)\smin X)$ is a 2-separation in $M\contract e$. Since $\si(M\contract e)$ is 3-connected, it follows that $|X - e| = 2$ or $E(M)\smin X = \{g, h\}$. In the former case, $X$ is a triangle of $M$, contradicting the choice of $T$. In the latter case, since $\{g\}, \{h\} \in \tangle$, we contradict Definition \ref{def:tangle}\eqref{it:tan3}.

Hence we may assume that for all $e\in X\smin f$, exactly one of $M\delete e$ and $M\contract e$ has $N$ as a minor.
	
	\begin{claim}
		There is an element $e \in X\smin f$ such that $M\delete e$ has $N$ as a minor and $X \smin e$ has no series pairs, or $M\contract e$ has $N$ as a minor and $X\smin e$ has no parallel pairs.
	\end{claim}
	\begin{subproof}
		Pick, possibly after dualizing, an element $e\in X\smin f$ such that $M\contract e$ has $N$ as a minor. If $M\contract e$ had no parallel pairs in $X\smin e$, then we would be done, so we can assume that $e$ is in a triangle $\{e,g,h\}\subseteq X$ in $M$. Assume $g\neq f$. Since $\{g,h\}$ is a parallel pair in $M\contract e$, it follows that $M\delete g$ has $N$ as a minor, and because $g$ is in no triad, $X\smin g$ contains no series pairs in $M\delete g$.
	\end{subproof}
	
	Now let $e$ be an element such that $M\contract e$ has $N$ as a minor and $X\smin e$ has no parallel pairs. Clearly $X\smin e$ also does not have series pairs in $M\contract e$. If $M\contract e$ is 3-connected, then the result follows. Otherwise $M\contract e$ has a 2-separation $(A,B)$ with $B\subseteq X$. By Lemma \ref{lem:simplecosimple2sep}, we have $|B| \geq 3$. Then Corollary \ref{cor:2sepminordelcon} with Lemma \ref{lem:2sepclosurecoclosuresmall} (if $f \not\in E(N)$) or Lemma \ref{lem:2sepminorintersect} (if $f \in E(N)$) imply the existence of an element $e' \in X$ such that both $M\contract e'$ and $M\delete e'$ have $N$ as a minor, a case we already dealt with. Duality now completes the proof.
\end{proof}

Next, we find a set of deletions and contractions:

\begin{theorem}\label{thm:delconset}
	Let $s$ be an integer, let $M$ be a 3-connected matroid, let $\tangle$ be a tangle of $M$ of order $\theta \geq 6$, and let $N$ be a minor of $M$ with no loops and coloops. If $\theta \geq 2s + t + 1$, then there are disjoint sets $C,D\subseteq E(M)\smin E(N)$ such that $M\contract C \delete D$ is 3-connected with $N$ as a minor, such that $\rank_\tangle(E(N)\cup C\cup D) = t + |C\cup D|$, and such that $|C\cup D| \geq s$.
\end{theorem}

To achieve this we use the following lemma:

\begin{lemma}\label{lem:deleteonetangle}
	Let $M$ be a 3-connected matroid, let $\tangle$ be a tangle of $M$ of order $\theta \geq 3$, let $N$ be a minor of $M$ with no loops and coloops, and let $H$ be a closed set of $M(\tangle)$ containing $E(N)$. If $\theta > \rank_\tangle(H)$, then there is an element $e\in E(M)\smin H$ such that one of $M\delete e$ and $M\contract e$ is 3-connected with $N$ as a minor.
\end{lemma}

\begin{proof}
	Suppose there is an element $e\in E(M)\smin H$ such that $e$ is on no long line of $M(\tangle)$. Let $(X, Y)$ be a 3-separation of $M$ with $e \in X$. If $X \in \tangle$, then $\rank_\tangle(X) = 2$, and therefore $e$ is contained in a long line of $M(\tangle)$, a contradiction. Hence we must have $Y \in \tangle$. If $e \in \closure_M(Y)$ or $e\in\coclosure_M(Y)$, then $Y\cup e$ is 3-separating and hence (by Lemma \ref{lem:crosstangle}) $Y\cup e \in \tangle$, and again $e$ is contained in a long line of $M(\tangle)$, a contradiction. It follows from Lemma \ref{lem:3sepgutscoguts} that both $M\contract e$ and $M\delete e$ are 3-connected. One of these has $N$ as a minor, and the result follows.
	
	Now pick $e\in E(M)\smin H$, and let $X$ be a long line containing $e$. Note that $X$ intersects $H$, and therefore $E(N)$, in at most one element. The result now follows from Theorem \ref{thm:contractunprotected}.	%
\end{proof}

With this in hand, the proof of Theorem \ref{thm:delconset} is no longer difficult.

\begin{proof}[Proof of Theorem \ref{thm:delconset}]
    Let $t := \rank_\tangle(E(N))$.
	Let $C,D\subseteq E(M)\smin E(N)$ be disjoint, such that $M\contract C \delete D$ is 3-connected with $N$ as a minor, such that $\rank_\tangle(E(N)\cup C\cup D) = t + |C\cup D|$, and such that $|C\cup D|$ is maximal. Suppose $|C\cup D| < s$. Define $H := \closure_\tangle(E(N)\cup C\cup D)$.
	
	Let $M' := M\contract C\delete D$, let $\theta' := \theta - |C\cup D|$, let $\tangle'$ be the tangle of $M'$ of order $\theta'$ inherited from $\tangle$, and let $H' := \closure_{\tangle'}(E(N))$. Then
	\begin{align*}
		\theta' & = \theta - |C\cup D| \\
		        & \geq 2s + t + 1 - (s-1)\\
		        & = s + t + 2\\
		        & \geq (\rank_{\tangle'}(H') + 1) + 2.
	\end{align*}
	Clearly $\theta' \geq 3$. But then Lemma \ref{lem:deleteonetangle} implies we can find an element $e\in E(M')\smin H'$ such that one of $M'\delete e, M'\contract e$ is 3-connected with $N$ as a minor. Since $e\not\in\closure_{\tangle'}(H')$, certainly $e\not\in\closure_\tangle(H)$, contradicting the maximality of $|C\cup D|$.
\end{proof}


The final lemma of this section deals with a rather specific case in which elements can be removed simultaneously.

\begin{lemma}\label{lem:manyfans}
	Let $M$ be a 3-connected matroid, let $\tangle$ be a tangle of $M$, let $N$ be a minor of $M$, and let $X_1, \ldots, X_r$ be long lines of $M(\tangle)$  with $\rank_\tangle(X_1\cup\cdots\cup X_r) = 2r$. Suppose the following properties hold for all $i\in \{1,\ldots,r\}$: 
	\begin{enumerate}
		\item $X_i\cap E(N) = \emptyset$;
		\item $X_i$ contains a maximal fan $F_i$ of length at least four;
		\item there is an element $e_i \in F_i$ such that $M\delete e_i$ is 3-connected with $N$ as a minor.
	\end{enumerate} 
	Then $M\delete \{e_1,\ldots,e_r\}$ is 3-connected with $N$ as a minor.
\end{lemma}

\begin{proof}
	We prove the result by induction on $r$, the case $r = 1$ being trivial. Let $r > 1$, and assume the result holds for all $r' < r$.	Consider $M' := M\delete e_r$, and let $\tangle'$ be the tangle of $M'$ inherited from $\tangle$. Pick any $i\in \{1,\ldots,r-1\}$. By Lemma \ref{lem:skew3seps} we have that $\closure_{\tangle'}(X_i) = X_i$. Moreover, since $e_r$ is not in the coclosure of $X_i$, the fan $F_i$ is still maximal in $M'$. Clearly $e_i$ is one of the ends of $F_i$, and then Lemma \ref{lem:endoffan} implies that $M'\delete e_i$ is 3-connected with $N$ as a minor.
	
	It follows that $M', \tangle', N, X_1, \ldots, X_{r-1}$ satisfy all the conditions of the lemma, and hence $M '' := M'\delete\{e_1,\ldots,e_{r-1}\}$ is 3-connected with $N$ as a minor, by induction. But $M'' = M\delete \{e_1,\ldots,e_r\}$, and the result follows.
\end{proof}


%


\section{The restoration graph} 
\label{sec:the_restoration_graph}

We know now that we can find sets $C$ and $D$ with $|C\cup D|$ large, such that $M\contract C \delete D$ is 3-connected with $N$ as a minor, but, for our main result, we require that either all elements are deleted or all elements are contracted. In the remainder of the paper, we will achieve this by studying subsets of $C\cup D$.

The following is a special case of \cite[Proposition 8.2.7]{ox2}.
\begin{lemma}\label{lem:almost3c}
	Let $e$ be an element of a matroid $M$. If $M\delete e$ is 3-connected but $M$ is not, then $e$ is either a loop, or a coloop, or in a parallel pair in $M$.
\end{lemma}

\begin{lemma}\label{lem:delcon4fan}
	Let $M$ be a matroid, $\tangle$ a tangle of $M$, and $\{c,d\}$ a $\tangle$-independent subset of $E(M)$ such that $M\contract c \delete d$ is 3-connected but $M\contract c$ is not. If $d$ is not in a parallel pair in $M$, then $M$ is 3-connected. Moreover, either $M\delete d$ is 3-connected or $c$ and $d$ are internal elements of a fan with size at least 4.
\end{lemma}

\begin{proof}
	By Lemma \ref{lem:tangleindep}, neither $c$ nor $d$ is a loop or coloop in any of $M$, $M\contract c$, and $M\delete d$. 
	Suppose that $M$ is not 3-connected. Let $(A,B)$ be a 2-separation of $M$, with $|A\smin\{c,d\}| \leq |B\smin\{c,d\}|$. Then $|A\smin \{c,d\}| \leq 1$, because otherwise $(A\smin \{c,d\},B\smin \{c,d\})$ would be a 2-separation of $M\contract c \delete d$. Since $\{c,d\} \in \tangle$ and $A\smin \{c,d\} \in \tangle$, it follows from Definition \ref{def:tangle}\eqref{it:tan3} that $E(M)\smin A \not \in \tangle$. Hence $A \in \tangle$. But then $|A\cap \{c,d\}| \leq 1$, since $\lambda_M(A) = 1 < \rank_{\tangle}(\{c,d\})$. It follows that $A$ is a series pair or a parallel pair containing exactly one of $c$ and $d$.	
	
	 Since $M\contract c$ is not 3-connected, $d$ has to be in parallel with some element $e$ in that matroid. In $M$ we find no parallel pair containing $d$, so $\{c,d,e\}$ must be a triangle of $M$. The element $c$ cannot be in any parallel pair of $M$, so $c$ must be in a series pair. But then $\{c,f\}$ is a series pair for some $f\in \{d,e\}$. Since $d,e \in \closure_M(\{c,f\})$, it follows that $\{c,d,e\}$ is 2-separating, contradicting the assumption that $\{c,d\}$ is $\tangle$-independent. We conclude that $M$ is 3-connected.
	
	For the second statement, suppose that $M\delete d$ is not 3-connected. Then $c$ must be in a series pair, say $\{c,f\}$. Since $M$ is 3-connected, we must have that $\{c,d,f\}$ is a triad of $M$. This implies that $e \neq f$ or $M \cong U_{2,4}$, and the result follows.
\end{proof}

\begin{lemma}\label{lem:undeletecompensate}
	Let $M$ be a 3-connected matroid, $\tangle$ a tangle of $M$, and $C,D$ disjoint subsets of $E(M)$ such that $C\cup D$ is $\tangle$-independent and $M\contract C \delete D$ is 3-connected. For each $d\in D$, either $M\contract C \delete (D\smin d)$ is 3-connected or there is an element $c\in C$ such that $M\contract (C\smin c)\delete (D\smin d)$ is 3-connected.
\end{lemma}

\begin{proof}
	Pick $d\in D$ such that $M\contract C \delete (D\smin d)$ is not 3-connected. Call the resulting matroid $M'$. It follows from Lemma \ref{lem:indeptangleminor} that $d$ is neither a loop nor a coloop of $M'$. It is impossible for $d$ to be in a series pair, so $d$ must be in a parallel pair, say with an element $e$.
	
	The set $\{d,e\}$ is not 2-separating in $M$, so there must be a circuit $Y$ with  $\{d,e\} \subsetneq Y \subseteq C\cup \{d,e\}$. Pick $c\in Y\cap C$. In $M\contract (C\smin c) \delete (D\smin d)$, we must have that $\{c,d,e\}$ is a triangle. Lemma \ref{lem:delcon4fan} now implies the result.
\end{proof}

It is convenient to keep track of deletions and contractions using a certain bipartite graph. Let us fix some notation. If $G = (V,E)$ is a graph, and $S\subseteq V$, then $G[S]$ is the \emph{induced subgraph} on $S$. For a vertex $v\in V$ we denote the set of vertices adjacent to $v$ but not equal to $v$ by $N(v)$. 

\begin{definition}\label{def:restoration}
	Let $M$ be a 3-connected matroid, and $C,D$ be disjoint subsets of $E(M)$ such that $M\contract C \delete D$ is 3-connected. The \emph{restoration graph} of $M$ with respect to $C$ and $D$, denoted by $R(M,C,D)$, is a bipartite graph with vertex set $C\cup D$ and edge set 
	\begin{align*}
		\{ cd : c\in C, d\in D, \textrm{ and } M\contract (C\smin c)\delete (D\smin d) \textrm{ is 3-connected}\}.
	\end{align*}
\end{definition}
Some more terminology: if $N = M\contract C \delete D$, and $Z\subseteq C\cup D$, then we say that $M\contract (C\smin Z) \delete (D\smin Z)$ was obtained from $N$ by \emph{restoring} $Z$. We say that an element $e \in C\cup D$ is \emph{privileged} if restoring $e$ yields a 3-connected matroid. 

If the set of vertices of a restoration graph is $\tangle$-independent for a tangle $\tangle$ of $M$, then it has many attractive properties. We list a few.

\begin{lemma}\label{lem:restgraphconn}
	Let $M$ be a 3-connected matroid, let $\tangle$ be a tangle of $M$, and let $C, D$ be disjoint subsets of $E(M)$ such that $C\cup D$ is $\tangle$-independent and $M\contract C \delete D$ is 3-connected. Then the restoration graph $R(M,C,D)$ has no isolated non-privileged vertices.
\end{lemma}

\begin{proof}
	This is an immediate consequence of Lemma \ref{lem:undeletecompensate} and its dual.
\end{proof}

\begin{lemma}\label{lem:restoresubgraph}
	Let $M$ be a 3-connected matroid, let $\tangle$ be a tangle of $M$, and let $C, D$ be disjoint subsets of $E(M)$ such that $C\cup D$ is $\tangle$-independent and $M\contract C \delete D$ is 3-connected. Let $G = R(M,C,D)$. Let $S\subseteq C\cup D$. Restoring $S$ yields a 3-connected matroid if and only if $G[S]$ has no isolated non-privileged vertices.
\end{lemma}

\begin{proof}
	Define $N := M\contract C \delete D$. Assume first that there is a set $S$ such that $G[S]$ has an isolated non-privileged vertex $d$, yet the matroid $M'$ obtained from $N$ by restoring $S$ is 3-connected. Using duality if necessary we may assume $d\in D$. The matroid obtained from $N$ by restoring $d$ is not 3-connected, so $d$ must be in a parallel pair with some element $e$ in that matroid. Clearly $\{d,e\}$ is not a parallel pair in $M'$, so there must be a circuit containing $d,e$, and at least one element $c\in C\cap S$. But then $c$ and $d$ satisfy all conditions of Lemma \ref{lem:delcon4fan}, and hence $cd$ is an edge of $G$, a contradiction.
	
	We will prove the converse by induction on the size of the set $S$ to be restored. The case $S = \emptyset$ is trivial, so we may assume $|S| \geq 1$. Pick $d \in S$ such that $d$ has minimum degree in the graph $G[S]$. If there is a choice, pick $d$ to be non-privileged, and consider $G[S\smin d]$. Using duality if necessary we may assume $d\in D$. Let $M'$ be the matroid obtained from $N$ by restoring $S$, and let $\tangle'$ be the tangle inherited from $\tangle$.
	
	First we assume that $G[S\smin d]$ has no isolated non-privileged vertices. By induction, restoring $S\smin d$ yields a 3-connected matroid. If $M'$ does have a 2-separation, then $d$ must be a loop or in parallel with another element in $E(M')$. The former cannot happen since $S$ is $\tangle'$-independent. Hence $d$ must be in parallel with an element $f$ of $E(M')$. Note that $f \not \in C\cup D$, because this again contradicts $\tangle'$-independence. Let $c$ be a neighbour of $d$ in $G[S]$, and let $N'$ be the matroid obtained from $N$ by restoring $\{c,d\}$. Then $N'$ is 3-connected. But $d,f \in E(N')$ and $N'$ is a minor of $M'$, so $\rank_{N'}(\{d,f\}) \leq \rank_{M'}(\{d,f\}) = 1$, a contradiction. It follows that restoring $S$ yields a 3-connected matroid.
	
	We may now assume that $G[S\smin d]$ has an isolated non-privileged vertex $c \in C$. In $G[S]$, there must be an edge $cd$, and both $c$ and $d$ have degree one. By induction, then, restoring $S\smin \{c,d\}$ yields a 3-connected graph. Suppose that $M'$ has a 2-separation $(A,B)$. The matroid $M'\contract c\delete d$ is 3-connected, so we must have $|A\smin\{c,d\}| \leq 1$ or $|B\smin\{c,d\}|\leq 1$. Assume, by relabelling if necessary, the former. Obviously $A \in \tangle$. Therefore $|A\cap S| \leq 1$. If $M''$ is the matroid obtained from $N$ by restoring $\{c,d\}$, then $A \subseteq E(M'')$, and $M''$ is a minor of $M'$. Hence $\lambda_{M''}(A) \leq 1$, contradicting the definition of the restoration graph.
\end{proof}

%


\section{The main result} 
\label{sec:the_main_result}

Before proving the main theorem, we find two structures in the restoration graph that will lead to the desired result. The first such structure, an imbalance between the sides, will be instrumental in our proof.

\begin{lemma}\label{lem:balance}
	Let $M$ be a matroid, let $\tangle$ be a tangle of $M$, and let $C,D\subseteq E(M)$ be such that $C\cup D$ is $\tangle$-independent, $M\contract C \delete D$ is 3-connected, and $|C| - |D| \geq k$. Then there is a subset $C'\subseteq C$ such that $|C'| \geq k$ and $M\contract C'$ is 3-connected.
\end{lemma}

\begin{proof}
	Let $G := R(M,C,D)$. Let $C'' \subseteq C$ be a minimal set such that each non-privileged $d \in D$ has a neighbour in $C''$. Clearly $|C''| \leq |D|$, and $G[C''\cup D]$ has no isolated non-privileged vertices. By Lemma \ref{lem:restoresubgraph}, restoring $C''\cup D$ yields a 3-connected matroid. This matroid is $M\contract (C\smin C'')$, and $|C| - |C''| \geq |C| - |D| \geq k$.
\end{proof}

We can use an \emph{induced} matching in the restoration graph to increase the imbalance between the sides, through the following lemma:

\begin{lemma}\label{lem:inducedmatching}
	Let $k$ be an integer, let $M$ be a matroid, let $\tangle$ be a tangle of $M$, let $N$ be a minor of $M$, let $t := \rank_\tangle(E(N))$, and let $C,D$ be disjoint subsets of $E(M)$ such that $\rank_\tangle(E(N)\cup C \cup D) = t + |C\cup D|$. If $R(M,C,D)$ contains an induced matching with at least $2k$	edges and no privileged vertices, then at least one of the following holds:
	\begin{enumerate}
		\item There is a set $C'\subseteq E(M)$ such that $\rank_\tangle(E(N)\cup C') = t + |C'|$, such that $|C'| \geq k$, and such that $M\contract C'$ is 3-connected with $N$ as a minor;
		\item There is a set $D'\subseteq E(M)$ such that $\rank_\tangle(E(N)\cup D') = t + |D'|$, such that $|D'| \geq k$, and such that $M\delete D'$ is 3-connected with $N$ as a minor.
	\end{enumerate}
\end{lemma}

\begin{proof}
	Define $G := R(M,C,D)$. By dualizing $M$ and $N$, and swapping $C$ and $D$ if necessary, assume $|D| \leq |C|$. If $|C| - |D| \geq k$, then the result follows from Lemma \ref{lem:balance}, so assume $|C| - |D| = r < k$. Let $H$ be a maximum-sized induced matching of $G$ with at least $2k$
	edges and no privileged vertices, and let $M'$ be the matroid obtained from $M$ by restoring $V(H)$. By Lemma \ref{lem:restoresubgraph}, $M'$ is 3-connected, $M'\delete d$ is not 3-connected for each $d\in D\cap V(H)$, and $M'\contract c$ is not 3-connected for each $c \in C\cap V(H)$. By Lemma \ref{lem:delcon4fan}, if $c$ and $d$ are adjacent vertices in the graph $H$, then they are internal elements of a fan $F$ of $M'$ of length at least 4. 
	
	Let $\tangle'$ be the tangle of $M'$ inherited from $\tangle$, and for each edge $cd \in E(H)$, let $X_{cd}$ be the long line of $M(\tangle')$ containing $c$ and $d$. Since $\rank_{\tangle'}(E(N)\cup\{c,d\}) = t + 2$, we have $X_{cd}\cap E(N) = \emptyset$. Moreover, $\rank_{\tangle'}(\bigcup_{cd \in E(H)} X_{cd}) = 2 |E(H)|$. For each $cd \in E(H)$, let $F_{cd}$ be the maximal fan of $M'$ containing $c$ and $d$, and let $x_{cd}$ be an end of the fan. By Lemma \ref{lem:endoffan}, $x_{cd}$ can be chosen such that one of $M'\delete x_{cd}$ and $M'\contract x_{cd}$ is 3-connected with $N$ as a minor. 
	
	Consider the set $S := \{x_{cd} : cd \in E(H)\}$. Let $S'\subseteq S$ be such that $M'\delete s$ is 3-connected for all $s\in S'$. Suppose $|S'| \geq k + r$. Define $C' := C\smin V(H)$, and $D' := (D\smin V(H)) \cup S'$. Then Lemma \ref{lem:manyfans} implies that $M\contract C'\delete D'$ is 3-connected with $N$ as a minor, and
	\begin{align*}
		|D'| - |C'| = |S'| + \left(|D| - \frac{|V(H)|}{2}\right) - \left(|C| - \frac{|V(H)|}{2}\right) = |S'| - r \geq k,
	\end{align*}
	so the result follows from Lemma \ref{lem:balance}. Similarly, if $S'\subseteq S$ is such that $M\contract s$ is 3-connected for all $s\in S'$, then the result follows if $|S'| \geq k - r$. But since $|S|\geq 2k$, one of these situations must hold, which completes our proof.
\end{proof}

Now we can state our main result. As mentioned in the introduction, it depends on the rank of $E(N)$ in $M(\tangle)$, rather than on the size of $N$.

\newcommand{\bd}{\ensuremath{20k + t - 13}}
\begin{theorem}\label{thm:mainres}
	Let $k$ be a nonnegative integer, let $M$ be a 3-connected matroid, let $\tangle$ be a tangle of $M$, let $N$ be a minor of $M$ with no loops and coloops, and let $t := \rank_{\tangle}(E(N))$. If the order of $\tangle$ is at least $\bd$, then there is a set $X \subseteq E(M)$ of size $k$ such that $\rank_\tangle(E(N)\cup X) = t + k$, and such that one of $M\delete X$ and $M\contract X$ is 3-connected with $N$ as a minor.
\end{theorem}

\begin{proof}
	By applying Theorem \ref{thm:delconset} with $s = 10k-7$ we can find sets $C,D \subseteq E(M)$ such that $\rank_{\tangle}(E(N)\cup C\cup D) = t + |C\cup D|$, such that $M\contract C \delete D$ is 3-connected with $N$ as a minor, and such that $|C\cup D| \geq 10k - 7$. Let $G := R(M,C,D)$ be the restoration graph, and let $|C| - |D| = r$. We will call $r$ the \emph{balance} of the restoration graph. If $|r| \geq k$, then we are done by Lemma \ref{lem:balance}, so we may assume this is not the case. We  partition the vertices of $G$ into disjoint subsets $P_1,P_2,Q_1,Q_2,T_1,T_2,U_1,U_2$, with sizes $p_1,p_2,q_1,q_2,t_1,t_2,u_1,u_2$ respectively, as follows.
	
	Let $P_1$ be the set of privileged vertices in $C$, and let $P_2$ be the set of privileged vertices in $D$. Let $Q_1$ be the vertices of $C$ that only have neighbours in $P_2$, and let $Q_2$ be the set of vertices of $D$ that only have neighbours in $P_1$. Let $C' := C\smin (P_1\cup Q_1)$, let $D' := D\smin (P_2\cup Q_2)$, and let $G' := G[C'\cup D']$. Let $R$ be the vertex set of a maximal matching in $G'$. Note that, by our choice of $Q_1$ and $Q_2$, no vertex of $G'$ is isolated, so all vertices in $V(G')\smin R$ have a neighbour in $R$.
	
	Let $S_1 \subseteq R\cap C'$ be a minimal set such that the set of neighbours $N(S_1)$ includes all vertices in $D'\smin R$. Clearly $|S_1| \leq |D'\smin R|$, and $|N(S_1)\cap R| \geq |S_1|$ because $R$ is a matching. Hence $|N(S_1)| \geq 2 |S_1|$. Now let $S_1'$ be a maximal set containing $S_1$ such that $|N(S_1')| \geq 2|S_1'|$, and define $U_2 := N(S_1')$. Let $T_2 := R\smin U_2$.
	
	Symmetrically, let $S_2 \subseteq R\cap D'$ be a minimal set such that $N(S_2)$ includes all vertices in $C'\smin R$. Let $S_2'$ be a maximal set containing $S_2$ such that $|N(S_2')| \geq 2|S_2'|$, and define $U_1 := N(S_2')$. Let $T_1 := R\smin U_1$. From the definitions it follows immediately that $P_1, Q_1, U_1, T_1$ partition $C$, and that $P_2,Q_2,U_2, T_2$ partition $D$. We will now bound the sizes of these sets.
	
	If $p_1 \geq k + r$, then restoring $P_1$ yields a 3-connected matroid having a restoration graph with balance $|C| - p_1 - |D| \leq |C| - |D| - (k+r) = - k$, and the result follows from Lemma \ref{lem:balance}. Similarly, if $p_2 \geq k - r$, then restoring $P_2$ yields a restoration graph with balance $k$. It follows that we may assume 
	\begin{align}
		p_1 + p_2 \leq 2 k - 2.\label{eq:p12}
	\end{align}
	
	Let $s_1 := |S_1'|$ and $s_2 := |S_2'|$. If $s_1 \geq k - r$, then restoring $S_1' \cup U_2$ yields a restoration graph with balance
	\begin{align}
		|C| - s_1 - (|D| - u_2) = r + u_2 - s_1 \geq r + s_1 \geq r + k - r = k,
	\end{align}
	and we can apply Lemma \ref{lem:balance} again. Likewise, if $s_2 \geq k + r$, then we can apply Lemma \ref{lem:balance} to the restoration graph obtained by restoring $S_2' \cup U_1$. It follows that we may assume
	\begin{align}
		s_1 + s_2 \leq 2 k - 2.\label{eq:s12}
	\end{align}
	Finally, if $u_2 - s_1 + q_2 - p_1 + p_2 \geq k - r$, then restoring $U_2 \cup S_1' \cup Q_2 \cup P_1 \cup P_2$ yields a restoration graph with balance
	\begin{align}
		|C| - s_1 - p_1 - (|D| - u_2 - q_2 - p_2) \geq k,
	\end{align}
	and we can apply Lemma \ref{lem:balance} again. Likewise, if $u_1 - s_2 + q_1 - p_2 + p_1 \geq k + r$, then we can apply Lemma \ref{lem:balance} to the restoration graph obtained by restoring $U_1 \cup S_2' \cup Q_2 \cup P_2 \cup P_1$. It follows that
	\begin{align}
		u_1 - s_2 + q_1 + u_2 - s_1 + q_2 \leq 2k - 2.\label{eq:u12}
	\end{align}
	
	Next we direct our attention to $T_1$ and $T_2$. Let $H_1$ be the subgraph of the matching $R$ containing all edges that meet $T_1$. Let $H_2$ be the subgraph of the matching $R$ containing all edges that meet $T_2$.
	\begin{claim}
		The matchings $H_1$ and $H_2$ are induced subgraphs of $G$.
	\end{claim}
	\begin{subproof}
		If some vertex $c \in V(H_2)\cap C'$ has degree at least 2, then $c$ can be added to $S_1'$, a contradiction. Hence all vertices in $V(H_2)\cap C'$ have degree exactly 1, and necessarily all vertices in $V(H_2)\cap D'$ have degree exactly 1. We omit the identical proof for $H_1$.
	\end{subproof}
	If $t_1 \geq 2k$ or $t_2 \geq 2k$, then our result follows from Lemma \ref{lem:inducedmatching}. Hence we may assume that
	\begin{align}
		t_1 + t_2 \leq 4k - 2. \label{eq:t12}
	\end{align}

	Adding \eqref{eq:p12}, \eqref{eq:s12}, \eqref{eq:u12}, and \eqref{eq:t12} we find
	\begin{align}
		|C| + |D| = p_1 + u_1 + q_1 + t_1 + p_2 + u_2 + q_2 + t_2 \leq 10k - 8.
	\end{align}
	But $|C| + |D| \geq 10k - 7$ by assumption, a contradiction.
\end{proof}

The theorem from the introduction is now easy to prove:

\begin{proof}[Proof of Theorem \ref{thm:mainresintro}]
	Let $l$ be the number of elements of $N$ each of which is neither a loop nor a coloop. By Lemma \ref{lem:connectedminor}, $M$ has a minor $N'$ such that $N'$ has $N$ as a minor, $N'$ has no loops and no coloops, and $|E(N')| \leq |E(N)| + l$. Clearly
	\begin{align}
		\bw(N') \leq \bw(N) + l \leq 2 |E(N)|.
	\end{align}
	The result now follows from Theorem \ref{thm:mainres} applied to $M$ and $N'$.
\end{proof}

%
%

As a possible direction for future research, one could hope for a bound of a different nature, namely one that is a function of $k$ and $\rank_{\tangle}(E(M)\smin E(N))$. Presumably such a bound would necessitate keeping only a minor isomorphic to $N$. However, the ideas from this paper do not seem to be suitable for proving such a result, and it is unclear if such a result has applications.


\paragraph{Acknowledgements} We thank the two anonymous referees for their many suggestions. The exposition improved considerably as a result.

\renewcommand{\Dutchvon}[2]{#1}
\bibliography{../matbib2012}
\bibliographystyle{plainnat}
\end{document}